\documentclass[12pt]{amsart}

\usepackage{amssymb}
\usepackage{amsmath}
\usepackage{amsfonts}
\usepackage{graphicx}
\usepackage{amsthm}
\usepackage{enumerate}
\usepackage[mathscr]{eucal}
\usepackage{mathrsfs}
\usepackage{verbatim}
\usepackage{booktabs}
\usepackage{amssymb}
\usepackage{color}

\makeatletter
\@namedef{subjclassname@2020}{%
\textup{2020} Mathematics Subject Classification}
\makeatother

\numberwithin{equation}{section}

\theoremstyle{plain}
\newtheorem{theorem}{Theorem}[section]
\newtheorem{lemma}[theorem]{Lemma}

\theoremstyle{plain}

\numberwithin{equation}{section}

\theoremstyle{remark}
\newtheorem{remark}[theorem]{Remark}

\allowdisplaybreaks[4]

\begin{document}

\date{\today}

\title
[zeros of ultraspherical Bessel derivatives]{A note on zeros of derivatives of  ultraspherical Bessel functions }

\author{Tao Jiang}
\address{School of Mathematics and Statistics\\
Anhui Normal University\\
Wuhu, 241002\\ P.R. China}
\email{tjiang@ahnu.edu.cn}

\thanks{This work is supported by the University Annual Scientific Research Plan of Anhui Province (2022AH050173) and the Talent  Cultivation Foundation of Anhui Normal University (2022xjxm036). }

\begin{abstract}
 For any  fixed $\nu\ge 0, \delta\in \mathbb R$ and $x>0$,  we  investigate the positive zeros of the  derivatives $j'_{\nu,\delta}(x)$  and   $y'_{\nu,\delta}(x)$, where
%\begin{equation*}
% \left(x^{-\delta}J_{\nu}(x)\right)'=x^{-\delta}\left(J_{\nu}'(x)-\delta x^{-1}J_{\nu}(x)\right)
%\end{equation*}
%and
%\begin{equation*}
% \left(x^{-\delta}Y_{\nu}(x)\right)'=x^{-\delta}\left(Y_{\nu}'(x)-\delta x^{-1}Y_{\nu}(x)\right).
%\end{equation*}
\begin{equation*}
 j_{\nu,\delta}(x)=x^{-\delta}J_{\nu}(x)\quad\text{and} \quad y_{\nu,\delta}(x)=x^{-\delta}Y_{\nu}(x).
\end{equation*}
We derive asymptotic expansions for their $k$-th positive zeros  as $k\rightarrow \infty$.
\end{abstract}



\maketitle
%%%%%%%%%%%%%%%%%%%%%%%%%%%%%%%%%%%%%%%%%%%%%%%%%%%%%%%%%%%%%%%%%%%%%%%%%%%%%%%%%%%%%%%%%%%%%%%%%%%%%%%%%%%%%%%%%%%%%%%%%%%%%%%%%%%
%%%%%%%%%%%%%%%%%%%%%%%%%%%%%%%%%%%%%%%%%%%%%%%%%%%%%%%%%%%%%%%%%%%%%%%%%%%%%%%%%%%%%%%%%%%%%%%%%%%%%%%%%%%%%%%%%%%%%%%%%%%%%%%%%%%
\section{Introduction}
For any fixed $\nu\ge 0, \delta\in \mathbb R$ and $x>0$,  we study the asymptotic expansions of the positive zeros of the following derivatives %$j'_{\nu,\delta}(x)$ and $y'_{\nu,\delta}(x)$,
\begin{align*}
j'_{\nu,\delta}(x)=x^{-\delta}(J'_{\nu}(x)-\delta x^{-1}J_{\nu}(x))
\end{align*}
and
\begin{equation*}
 y'_{\nu,\delta}(x)=x^{-\delta}(Y'_{\nu}(x)-\delta x^{-1}Y_{\nu}(x)),
\end{equation*}
where \begin{equation*}
 j_{\nu,\delta}(x)=x^{-\delta}J_{\nu}(x)\quad \text{and}\quad y_{\nu,\delta}(x)=x^{-\delta}Y_{\nu}(x)
\end{equation*}
are the \emph{ultraspherical Bessel functions}, and  $J_{\nu}(x)$ and $Y_{\nu}(x)$ are the Bessel functions of the first and second kind of order $\nu$.

The study of zeros of Bessel functions and their derivatives is significant in  a variety of physical and mathematical problems. One simple application  arises
in  acoustics and electromagnetic theory, where they  help describe the vibrational modes of circular membranes and resonant frequencies in cylindrical cavities, as noted by Gray and  Mathews in \cite{GM:1895}.
In addition, these zeros are crucial for spectral theory, particularly in eigenvalue problems involving the Laplacian in bounded domains.
%Sometimes it is very helpful to know properties of their zeros.
%The study of the behaviour of the zeros of $j_{\nu,\delta}'(x)$ and $y_{\nu,\delta}'(x)$ arise regularly in a variety of physical and mathematical problems.
%It can be traced back to Lamb in \cite{Lams1884} examining this problem in his investigation of induced electric currents in a cylinder placed transverse to magnetic field lines.
%Also,
For instance, the positive  zeros of $j_{\nu,\delta}(x)$ and  $j_{\nu,\delta}'(x)$ when $\nu=n+d/2-1$ and $\delta=d/2-1$ with $ d\ge 2$ and $n\in \mathbb N$ are related to the eigenvalues of the Dirichlet and Neumann Laplacians on the unit  ball (disk when $d=2$) in $\mathbb R^d$, respectively.
 These zeros play a role in the study of Weyl's law for the  Laplacian associated with the unit ball.

The positive zeros of $ j_{\nu,\delta}(x)$ and $ y_{\nu,\delta}(x)$ are the same as the zeros of $ J_{\nu}(x)$ and $ Y_{\nu}(x)$, and formulas for the latter can be found in \cite[P. 371]{abram:1972}.

There are, however, no simple relations connecting the zeros of the derivatives $ j_{\nu,\delta}'(x)$ and $ y_{\nu,\delta}'(x)$.
In \cite[Table~III]{olver:1960}, Olver  presented the  zeros  of  derivatives  of the spherical Bessel functions $ j'_{n+1/2,1/2}(x)$ and $y'_{n+1/2,1/2}(x)$ for $n = 0,1,2,\ldots, 20$   respectively,  with eight-decimal-place accuracy.
Ifantis and Siafarikas in \cite{IS:1988} and Lorch and Szego in \cite{LS:1994}  proved various monotonicity properties of the positive zeros of $j_{\nu,\delta}'(x)$,
yielding lower and upper bounds for these zeros.
 Recently, using the modulus and phase functions, Filonov, Levitin, Polterovich and Sher in \cite{FLPS:2024} obtained an implicit expression  of the positive zeros of  $j_{\nu,\delta}'(x)$ and $y_{\nu,\delta}'(x)$ in terms of inverse values of some elementary functions and derived  bounds  of the zeros.

For the asymptotic expansions of positive zeros,
McMahon in \cite{mcmahon:1894} made significant contributions on  $J_{\nu}(x)$, $Y_{\nu}(x)$, $J_{\nu}'x)$, $Y_{\nu}'(x)$ and  other related Bessel functions. See \cite[P. 371 and P. 441]{abram:1972} or \eqref{m1} and \eqref{m2} in Remark \ref{r1} for  more details.
%detailed expressions of the positive zeros of the derivatives of Bessel functions and the derivatives of spherical Bessel functions.
In \cite{GJWY:2024}, Guo, Wang, Yang and the present author  derived a one-term asymptotic expansion of the  positive zeros of  $j'_{\nu, \delta}(x)$. In this paper, we provide  McMahon-type expansions of the positive zeros of $j'_{\nu, \delta}(x)$ and $y'_{\nu, \delta}(x)$.
Let $a_{\nu,\delta,k}'$ and $b_{\nu,\delta,k}'$ be the $k$-th positive zero (in ascending order of magnitude) of $ j'_{\nu,\delta}(x)$ and $y'_{\nu,\delta}(x)$ respectively; $x=0$ is counted as the first zero of $j'_{\nu,\delta}(x)$ when $\nu \le |\delta|$.
We then have
\begin{theorem}\label{mainresult}
For any fixed $\nu\geq 0$ and  $\delta\in\mathbb{R}$, if  $k$ is sufficiently large, then
 \begin{align}\label{expansions}
 \begin{split}
& a_{\nu,\delta,k}',\,b_{\nu,\delta,k}'\\
&\thicksim\beta'\!-\!\frac{\mu\!+3+8\delta}{8\beta'}\!-\frac{4\!\left(7\mu^2+\!(82+144\delta)\mu-9+144\delta+192\delta^2-128\delta^3\right)}{3(8\beta')^3}\\
& \quad+\frac{C_{\nu,\delta}}{(8\beta')^5}+\frac{\tilde{C}_{\nu,\delta}}{(8\beta')^7}+\cdots,
 \end{split}
 \end{align}
where $\beta'=(k+\frac{\nu}{2}-\frac{3}{4})\pi$ for $a_{\nu,\delta,k}'$, $\beta'=(k+\frac{\nu}{2}-\frac{1}{4})\pi$ for $b_{\nu,\delta,k}'$,\, $\mu=4\nu^2$,
\begin{align*}
&C_{\nu,\delta}=-\frac{32}{15}\Big(83\mu^3+(2075+2920\delta)\mu^2-(3039-9040\delta-10560\delta^2\\
&\,+4480\delta^3)\mu+3537+1800\delta+8640\delta^2+1920\delta^3-12800\delta^4+3072\delta^5\Big),
\end{align*}
and
\begin{align*}
&\tilde{C}_{\nu,\delta}=-\frac{64}{105}\Big(6949\mu^4+(296492+356832\delta)\mu^3
-(1248002-2194080\delta\\
&\quad-2298240\delta^2+743680\delta^3)\mu^2+(7414380+696864\delta+5295360\delta^2\\
&\quad+1917440\delta^3-4945920\delta^4+946176\delta^5)\mu-5853627+913248\delta\\
&\quad+1330560\delta^2+2338560\delta^3
-3225600\delta^4-4902912\delta^5+3555328\delta^6\\
&\quad-491520\delta^7\Big).
\end{align*}
\end{theorem}

\begin{remark}\label{r1}
If we take the special case $\delta=0$, then we obtain the McMahon's expansion for the positive zeros of derivatives of the Bessel functions, see \cite[P. 371]{abram:1972}
\begin{align}\label{m1}
\begin{split}
&a_{\nu,0,k}', b_{\nu,0,k}' \\
&\thicksim \beta' - \frac{\mu\! +\! 3}{8\beta'} - \frac{4(7\mu^2 \!+ 82\mu -\! 9)}{3(8\beta')^3} - \frac{32(83\mu^3 + 2075\mu^2-3039\mu + 3537)}{15(8\beta')^5} \\
&\quad- \frac{64(6949\mu^4 +\! 296492\mu^3 -\! 1248002\mu^2\! + 7414380\mu - \! 5853627)}{105(8\beta')^7}+\cdots,
\end{split}
\end{align}
where $\beta' = (k + \frac{\nu}{2} - \frac{3}{4})\pi$ for $a_{\nu,0,k}'$, $\beta' = (k + \frac{\nu}{2} - \frac{1}{4})\pi$ for $b_{\nu,0,k}'$, and  $\mu=4\nu^2$.

If we take the special case $\delta=\frac{1}{2}$ and $\nu=n+\frac{1}{2}$, then we obtain the McMahon's expansion for the positive zeros of derivatives of the spherical Bessel functions,
 see  \cite[P. 441]{abram:1972}
\begin{align}\label{m2}
\begin{split}
  &a_{n+\frac{1}{2},\frac{1}{2},k}',\, b_{n+\frac{1}{2},\frac{1}{2},k}'\\
  &\thicksim\!\beta'\!-\!\frac{\mu\!+\!7}{8\beta'}\!-\!\frac{4(7\mu^2\!+154\mu+95)}{3(8\beta')^3}\!-\!\frac{32(83\mu^3+3535\mu^2+3561\mu+6133)}{15(8\beta')^5}\\
  &\quad-\frac{64(6949\mu^4+474908\mu^3+330638\mu^2\!+9046780\mu-5075147)}{105(8\beta')^7}+\cdots,
\end{split}
\end{align}
where $\beta'=(k+\frac{n}{2}-\frac{1}{2})\pi$ for $a_{n+\frac{1}{2},\frac{1}{2},k}'$, $\beta'=(k+\frac{n}{2})\pi$ for $b_{n+\frac{1}{2},\frac{1}{2},k}'$ and $\mu=4(n+1/2)^2$.
\end{remark}
\begin{remark}
In Theorem \ref{mainresult}, we only provide the first five terms of the expansions for $a_{\nu,\delta,k}'$ and $b_{\nu,\delta,k}'$(which already appear quite complicated).
However, it can be seen from the proof below that we can  obtain  the expansions to arbitrary order of $\frac{1}{\beta'}$.
\end{remark}

\emph{Notations:} For functions $f$ and $g$ with $g$ taking nonnegative real values,
$f\lesssim g$ means $|f|\leqslant Cg$ for some constant $C$. If $f$
is nonnegative, $f\gtrsim g$ means $g\lesssim f$. The notation
$f\asymp g$ means that $f\lesssim g$ and $g\lesssim f$. The Landau notation $f=O(g)$ is equivalent to $f\lesssim g$. If we write a subscript (for instance $O_{\nu}$), we emphasize that the implicit constant depends on that specific subscript. A series $\sum_{l=0}^{\infty}a_lx^{-l}$ is said to be an asymptotic expansion of a function $f(x)$ if
\begin{equation*}
  f(x)-\sum_{l=0}^{n-1}a_lx^{-l}=O(x^{-n}) \quad \text{as} \quad x\rightarrow\infty
\end{equation*}
for every $n=1,2,3,\ldots$, and we write
\begin{equation*}
  f(x)\thicksim\sum_{l=0}^{\infty}a_lx^{-l}.
\end{equation*}
%%%%%%%%%%%%%%%%%%%%%%%%%%%%%%%%%%%%%%%%%%%%%%%%%%%%%%%%%%%%%%%%%%%%%%%%%%%%%%%%%%%%%%%%%%%%%%%%%%%%%%%%%%%%%%%%%%%%%%%%%%%%%%%%%%%%%%%%%%%%%%%%%%%%%%%%%%%%%%%%%%%%%%%%%%%%%
\section{one-term asymptotics}
In this section, we first present some asymptotics of $Y_{\nu}(x)$ and $Y_{\nu}'(x)$ under  different assumpitons on $x$ and $\nu$, and consequently for $y_{\nu,\delta}'(x)$. We then establish the one-term asymptotics of $a_{\nu,\delta,k}'$ and $b_{\nu,\delta,k}'$. The asymptotics \eqref{ynasy1}-\eqref{ynasyNC3} are standard;  see for example \cite[Lemmas A.1 and A.2]{GJWY:2024}. Their proofs rely on the theory of oscillatory integrals and Olver's asymptotic expansions of the Bessel functions. Throughout this paper, we denote  $g(x)=\left(\sqrt{1-x^2}-x\arccos x\right)/\pi$ for $x\in [0,1]$.

%\begin{lemma}
%	For any $c>0$ and $\nu\ge 0$, if $x\geq \max\{(1+c)\nu, 10\}$ then
%	\begin{equation*}
%		j_{\nu,\delta}'(x)=- \frac{2^{1/2}\left(x^2-\nu^2\right)^{1/4}}{\pi^{1/2}x^{1+\delta}} \left(\sin\left( \pi
%		xg\left(\frac{\nu}{x}\right)-\frac{\pi}{4}\right)+O\left(x^{-1}\right)\right).
%	\end{equation*}
%
%For any sufficiently small $c>0$ and sufficiently large $\nu$, if $\nu<x<(1+c)\nu$ then
%\begin{equation*}
%	j_{\nu,\delta}'(x)=\frac{-2^{2/3}\left(x^2-\nu^2\right)^{1/4}}{(3\pi)^{1/6}x^{1+\delta}\left(x g\left(\nu/x \right)\right)^{1/6}}
%	\!\left(\!\mathrm{Ai}'\!\left(\!-\!\left(\!\frac{3\pi}{2} x g\!\left(\frac{\nu}{x} \right)\!\right)^{\!\frac{2}{3}}\right)\!+O\left(\nu^{-\frac{2}{3}}\right)\!\right)
%\end{equation*}
%when $xg(\nu/x)\leq C$ for some large constant $C>0$ , and
%	\begin{equation*}
%		j_{\nu,\delta}'(x)=\frac{-2^{1/2}\left(x^2-\nu^2\right)^{1/4}}{\pi^{1/2}x^{1+\delta}} \left(\sin\left( \pi
%		xg\!\left(\frac{\nu}{x}\right)-\frac{\pi}{4}\right)+O\left(\!\left(xg\left(\frac{\nu}{x}\right)\right)^{\!-1}\right)\!\right)
%	\end{equation*}
%when $xg(\nu/x)>C$.
%\end{lemma}
%\begin{proof}
%  See \cite[Lemma 4.8]{GJWY:2024}.
%\end{proof}

\begin{lemma}\label{lemma1}
For any $c>0$ and $\nu\ge 0$, if $x\ge \max\{(1+c)\nu,10\}$ then
 \begin{equation}\label{ynasy1}
Y_{\nu}(x)=\frac{2^{1/2}}{\pi^{1/2}\left(x^2-\nu^2\right)^{1/4} }\left(\sin\left( \pi
xg\left(\frac{\nu}{x}\right)-\frac{\pi}{4}\right)+O_{c}\left(x^{-1}\right)\right),
\end{equation}
\begin{equation}\label{ynasyNC1}
Y_{\nu}'(x)=\frac{2^{1/2}\left(x^2-\nu^2\right)^{1/4}}{\pi^{1/2}x}\left(\sin\left( \pi
xg\left(\frac{\nu}{x}\right)+\frac{\pi}{4}\right)+O_{c}\left(x^{-1}\right)\right).
\end{equation}
\end{lemma}

\begin{proof}
We only  prove the asymptotics of $Y_{\nu}'(x)$, as the proof for $Y_{\nu}(x)$ is similar.

It is easy to check that if  $x\ge \max\{(1+c)\nu,10\}$, we can obtain $ Y_{\nu}'(x)$ by  differentiating under the integral sign of the integral representation of $Y_{\nu}(x)$. So we have
 \begin{equation*}
 \begin{split}
 Y_{\nu}'(x)=\mathrm{Re}(I_{\nu}(x))+L_{\nu}(x),
 \end{split}
 \end{equation*}
where
\begin{equation*}
  I_{\nu}(x):=\frac{1}{\pi}\int_{0}^{\pi}e^{ix\phi(\theta)}\sin\theta\,\mathrm{d}\theta
\end{equation*}
with the phase function
\begin{equation*}
  \phi(\theta)=\phi(x,\nu,\theta)=\sin\theta-\frac{\nu}{x}\theta
\end{equation*}
and
\begin{equation*}
  L_{\nu}(x):=\frac{1}{\pi}\int_{0}^{\infty}\left(e^{\nu t}+e^{-\nu t}\cos(\nu \pi)\right)e^{-x\sinh t}\sinh t\,\mathrm{d}t.
\end{equation*}
%Since for $x\ge \max\{(1+c)\nu,10\}$, the Weierstrass M-test ensures uniform convergence of $L_{\nu}(x)$. This permits differentiating under the integral sign to obtain the above expression for $Y_{\nu}'(x)$.
Obviously, applying  integration by parts twice yields
\begin{equation*}
  L_{\nu}(x)=O_c\left(x^{-2}\right).
\end{equation*}

For the integral $I_{\nu}(x)$, we apply the method of stationary
phase to  it in a sufficiently small neighborhood of the critical point $\arccos(\nu/x)$. This yields the contribution
\begin{equation*}
  \frac{2^{1/2}(x^2-\nu^2)^{1/4}}{\pi^{1/2}x}e^{i\left(\pi xg(\frac{\nu}{x})-\frac{\pi}{4}\right)}(1+O_{c}(x^{-1})).
\end{equation*}
The contribution from the part of the integral  away from the critical point  is at most $O_{c}(x^{-2})$, obtained by applying integration by parts twice. Then we have \eqref{ynasyNC1} by taking the real part of $I_{\nu}(x)$.
This completes the proof.
\end{proof}

\begin{lemma}\label{lemma2}
 For any $c>0$ and all sufficiently large $\nu$, if $\nu<x<(1+c)\nu$ then
 \begin{equation}\label{ynasy2}
Y_{\nu}(x)=\frac{-(12\pi xg(\nu/x) )^{1/6}}{(x^2-\nu^2)^{1/4}}\left(\mathrm{Bi}\!\left(\!-\!\left(\frac{3\pi}{2} xg\!\left(\frac{\nu}{x}\right)\!\right)^{\!2/3}\right)+O\left(\nu^{-4/3}\right)\right),
\end{equation}
 \begin{equation}\label{ynasyNC2}
Y_{\nu}'(x)=\frac{2\left(x^2-\nu^2\right)^{1/4}}{\left(12\pi xg\left(\nu/x \right)\right)^{1/6}x}
	\!\left(\!\mathrm{Bi}'\!\left(\!-\!\left(\!\frac{3\pi}{2} x g\!\left(\frac{\nu}{x} \right)\!\right)^{\!2/3}\right)\!+O\left(\nu^{-2/3}\right)\!\right)
\end{equation}
 when $xg(\nu/x)\leq C$ for a large constant $C>0$, and
\begin{equation}\label{ynasy3}
Y_{\nu}(x)=\frac{2^{1/2}}{\pi^{1/2} \left(x^2-\nu^2\right)^{1/4}} \left(\sin\left(\pi xg\!\left(\frac{\nu}{x}\right)-\frac{\pi}{4}\right)
+O\left(\!\left(xg\left(\frac{\nu}{x}\right)\right)^{\!-1}\right)\!\right),
\end{equation}
\begin{equation}\label{ynasyNC3}
Y_{\nu}'(x)=\frac{2^{1/2}\left(x^2-\nu^2\right)^{1/4}}{\pi^{1/2}x} \left(\sin\left(\pi xg\!\left(\frac{\nu}{x}\right)+\frac{\pi}{4}\right)
+O\left(\!\left(xg\left(\frac{\nu}{x}\right)\right)^{\!-1}\right)\!\right)
\end{equation}
when $xg(\nu/x)>C$. Here the implicit constants only depends on $c$ and $C$.
\end{lemma}

\begin{proof}
The asymptotics are consequences of  Olver's asymptotic expansions of  Bessel functions  and well-known asymptotics of Airy functions. We only  prove the asymptotics of $Y_{\nu}'(x)$, as the proof for $Y_{\nu}(x)$ is similar.

For sufficiently large $\nu$, applying  Olver’s asymptotics of $Y_{\nu}'(\nu z)$ ( see \cite[P. 369]{abram:1972}) we obtain
\begin{equation}\label{Olverexpansion}
  Y_{\nu}'(x)= \! \frac{2}{x/\nu}\!\!\left(\!\frac{1\!-\!(x/\nu)^2}{4\zeta}\!\right)^{\!1/4}\!\!\left(
  \frac{\mathrm{Bi}'(\nu^{2/3}\zeta)}{\nu^{2/3}}\!\left(1\!+O\!\left(\nu^{-2}\right)\right)
  +\frac{\mathrm{Bi}(\nu^{2/3}\zeta)}{\nu^{4/3}}O(1)\!\right)
\end{equation}
with $\zeta=\zeta(x)$ determined by
\begin{equation*}
\frac{2}{3}(-\zeta)^{3/2}=\sqrt{\left(\frac{x}{\nu}\right)^2-1}-\arccos\frac{\nu}{x}
\end{equation*}
as $1<x/\nu<1+c$. Notice that
\begin{equation*}
\nu^{2/3}\zeta=-\left(\frac{3\pi}{2}xg\left(\frac{\nu}{x}\right)\right)^{2/3}.
\end{equation*}
Then for a large constant $C>0$, if  $xg(\nu/x)\leq C$, we obtain \eqref{ynasyNC2} by  continuity of $\mathrm{Bi}(-z)$ and $\mathrm{Bi}'(-z)$.

If $xg(\nu/x)> C$, combining \eqref{Olverexpansion}, the  asymptotics of $\mathrm{Bi}(-z)$ and $\mathrm{Bi}'(-z)$ (see \cite[P. 449]{abram:1972}) and the fact $xg(\nu/x)\lesssim \nu$ yields \eqref{ynasyNC3}. This completes the proof.
\end{proof}

\begin{lemma}
For any $c>0$ and $\nu\ge 0$, if $x\geq \max\{(1+c)\nu, 10\}$ then
\begin{equation}\label{y1}
y_{\nu,\delta}'(x)= \frac{2^{1/2}\left(x^2-\nu^2\right)^{1/4}}{\pi^{1/2}x^{1+\delta}} \left(\sin\left( \pi xg\left(\frac{\nu}{x}\right)+\frac{\pi}{4}\right)+O_c\left(x^{-1}\right)\right).
\end{equation}

For any sufficiently small $c>0$ and sufficiently large $\nu$, if $\nu<x<(1+c)\nu$ then
\begin{equation}\label{y2} y_{\nu,\delta}'(x)=\frac{2\left(x^2-\nu^2\right)^{1/4}}{\left(12\pi x g\left(\nu/x \right)\right)^{1/6}x^{1+\delta}}
	\!\left(\!\mathrm{Bi}'\!\left(\!-\!\left(\!\frac{3\pi}{2} x g\!\left(\frac{\nu}{x} \right)\!\right)^{\!\frac{2}{3}}\right)\!+O\left(\nu^{-\frac{2}{3}}\right)\!\right)
\end{equation}
when $xg(\nu/x)\leq C$ for a large constant $C>0$, and
\begin{equation}\label{y3}	y_{\nu,\delta}'(x)=\frac{2^{1/2}\left(x^2-\nu^2\right)^{1/4}}{\pi^{1/2}x^{1+\delta}} \left(\sin\left( \pi xg\!\left(\frac{\nu}{x}\right)+\frac{\pi}{4}\right)+O\left(\!\left(xg\left(\frac{\nu}{x}\right)\right)^{\!-1}\right)\!\right)
\end{equation}
when $xg(\nu/x)>C$.  Here the implicit constants only depends on $c$ and $C$.
\end{lemma}

\begin{proof}
From Lemmas \ref{lemma1} and  \ref{lemma2}, we see that the term $\delta x^{-1}Y_{\nu}(x)$ can be absorbed into the error terms of the asymptotics of  $Y_{\nu}'(x)$. For the case of  \eqref{y3}, however, one needs to note that
\begin{equation*}\label{xg(x/nu)}
  xg(\nu/x)\asymp \nu^{-1/2}(x-v)^{3/2},
\end{equation*}
which follows  from the asymptotic expansion
\begin{equation*}
  \sqrt{z^2-1}-\arccos\frac{1}{z}=\frac{2\sqrt{2}}{3}(z-1)^{3/2}+O\left((z-1)^{5/2}\right)
\end{equation*}
as $z\rightarrow 1^+$.
\end{proof}

\begin{lemma}\label{largenu}
There exists a constant $c\in (0,1)$ such that for any $\delta\in\mathbb{R}$ and all sufficiently large $\nu$,
\begin{equation}\label{case1}
	b'_{\nu, \delta, k}g\left(\frac{\nu}{b'_{\nu, \delta, k}}\right)=\left\{
\begin{array}{ll}
k-\frac{1}{4}+O((\nu+k)^{-1}),  & \textrm{if  $b'_{\nu, \delta, k}\ge (1+c)\nu$,}\\
\\
k-\frac{1}{4}+O(k^{-1}),  & \textrm{if $\nu<b'_{\nu, \delta, k}< (1+c)\nu$.}
\end{array}\right.
%=k-\frac{1}{4}+O((\nu+k)^{-1}),
\end{equation}
% when $b'_{\nu, \delta, k}\ge (1+c)\nu$, and
%\begin{equation}\label{case2}
%b'_{\nu, \delta, k}g\left(\frac{\nu}{b'_{\nu, \delta, k}}\right)=k-\frac{1}{4}+O(k^{-1}),
%\end{equation}
%when $\nu<b'_{\nu, \delta, k}< (1+c)\nu$.
\end{lemma}

\begin{proof}
We will study zeros of $y_{\nu,\delta}'(x)$ in the interval $(\nu,(s+\nu/2+1/2)\pi)$ for large integer $s>\nu^3 $. Let $h_{\nu}(x)=xg(\nu/x)$.
It can be  checked that the function
\begin{equation*}
h_\nu: [\nu, \infty)\rightarrow [0, \infty)
\end{equation*}
is continuous, strictly increasing and mapping $(\nu, (s+\nu/2+1/2)\pi)$ onto $(0, s+1/2+O(\nu^{2}/(s+\nu)))$. Therefore for each integer $1\leq k\leq s$ there exists an interval $(a_k, b_k)\subset (\nu, (s+\nu/2+1/2)\pi)$ such that $h_\nu$ maps $(a_k, b_k)$ to $(k-1/2, k)$  bijectively. All these intervals $(a_k, b_k)$'s are clearly disjoint.

If $\nu$ is sufficiently large then for each $1\leq k\leq s$
\begin{equation}
y_{\nu,\delta}'(a_k)y_{\nu,\delta}'(b_k)<0.\label{IVT-conditionNC}
\end{equation}
To prove this, we first fix a sufficiently large constant $\tilde C$ such that if $h_{\nu}(x)>\tilde C$, then the error term $O\left((h_{\nu}(x))^{-1}\right)$ in \eqref{y3} is less than
$10^{-10}$. Then we  consider sufficiently large $\nu$ such that the error term $O_{c}\left(x^{-1}\right)$ in \eqref{y1} is less than $10^{-10}$. If
\begin{equation}\label{h}
h_{\nu}((a_k, b_k))=(k-1/2, k)\subset (0,2\tilde C],
\end{equation}
we use \eqref{y2} to prove \eqref{IVT-conditionNC}, otherwise we use \eqref{y1} and \eqref{y3}. In the former case when \eqref{h} holds, there are at most $\lfloor 2\tilde C\rfloor$ choices of $k$.  The sign of $y_{\nu,\delta}'(x)$ depends on that of
\begin{equation}\label{Bi'}
\mathrm{Bi}'\left(-\left(3\pi h_{\nu}(x)/2 \right)^{2/3}\right)+O\left(\nu^{-2/3}\right).
\end{equation}
We denote by
$t_k$  the $k$-th zero of  $\mathrm{Bi}'(-x)$. We can derive that
\begin{equation*}
t_k=\left(\frac{3\pi}{2}\left(k-\frac{1}{4}+\beta'_k\right)\right)^{2/3}
\end{equation*}
with $|\beta'_k|<\left(e^{7/(9\pi)}-1\right)/2$ by using results in \cite[P. 214 and 405]{olver:1997}). Therefore
\begin{align}
t_k\in &\left(\left(\frac{3\pi}{2}\left(k-0.4\right)\right)^{2/3}, \left(\frac{3\pi}{2}\left(k-0.1\right)\right)^{2/3} \right) \nonumber \\
&\subsetneq \left(\left(\frac{3\pi}{2}h_\nu(a_k)\right)^{2/3}, \left(\frac{3\pi}{2}h_\nu(b_k)\right)^{2/3} \right). \label{theorem-2NC}
\end{align}
Since $\textrm{Bi}'(-x)$ oscillates around zero for positive $x$ and the intervals in \eqref{theorem-2NC} are pairwise disjoint, the signs of \eqref{Bi'} at $x=a_k$ and $b_k$ must be opposite whenever $\nu$ is sufficiently large, which ensures \eqref{IVT-conditionNC} in this case.
In the latter case when \eqref{h} fails, we must have $h_{\nu}((a_{\nu,k}, b_{\nu,k}))\subset (\tilde C,\infty)$. Thus \eqref{IVT-conditionNC} follows immediately from \eqref{y1} and \eqref{y3},  noting that
\begin{equation*}
  \sin\left(\pi h_{\nu}(a_k)+\frac{\pi}{4}\right)\sin\left(\pi h_{\nu}(b_k)+\frac{\pi}{4}\right)=-\frac{1}{2}<0.
\end{equation*}

Now we continue to prove the lemma. By the intermediate value
theorem, \eqref{IVT-conditionNC} implies that there exists at least one zero of $y_{\nu,\delta}'(x)$ in $(a_k, b_k)$ for each $1\le k\le s$.
On the other hand \cite[Lemmas 3.1 and 3.2]{FLPS:2024} show that if $s$ is sufficiently large, there are exactly $s$ zeros of $y_{\nu,\delta}'(x)$ in the interval $(0, (s+\nu/2+1/2)\pi)$. Hence $y_{\nu,\delta}'(x)$ has one and only one zero in each $(a_k, b_k)\subset (\nu, (s+\nu/2+1/2)\pi)$, which is $b_{\nu,\delta,k}'$ satisfying
\begin{equation}\label{hb}
  h_{\nu}(b_{\nu,\delta,k}')=b_{\nu,\delta,k}'g(\nu/b_{\nu,\delta,k}')\in \left(k-1/2,k\right).
\end{equation}
A direct observation is that when $\nu$ is sufficiently large, the lower bound for $b_{\nu,\delta,k}' (k\in \mathbb N)$ satisfies $b_{\nu,\delta,k}'>\nu$ and  $b_{\nu,\delta,k}'\gtrsim k$ since $0<g(x)\le 1/\pi$.

If $b_{\nu,\delta,k}'\ge (1+c)\nu$, then \eqref{case1} in this case follows from \eqref{y1}, \eqref{hb} and the lower bound of $b_{\nu,\delta,k}'$.
If  $\nu<b_{\nu,\delta,k}'< (1+c)\nu$ and $h_{\nu}(b_{\nu,\delta,k}')< 2\tilde C$, then $1\le k < 2\tilde C+1$. Hence \eqref{case1} in this case follows trivially from \eqref{hb}.
If $\nu<b_{\nu,\delta,k}'< (1+c)\nu$ and $h_{\nu}(b_{\nu,\delta,k}')\ge\tilde  C$, then combining \eqref{y3} and \eqref{hb} yields
\begin{equation*}
  h_{\nu}(b_{\nu,\delta,k}')=k-\frac{1}{4}+O\left((b_{\nu,\delta,k}')^{-1}\right).
\end{equation*}
This formula also implies $ h_{\nu}(b_{\nu,\delta,k}')\asymp k$, hence \eqref{case1} in this case.
This completes the proof.
\end{proof}

\begin{lemma}\label{smallnu}
For any $\delta\in\mathbb{R}$ and $V\in\mathbb{N}$ there exists a constant $K>0$ such that if $0\leq \nu\leq V$ and $k\geq K$ then
\begin{equation}\label{eqsmallnu}
		b'_{\nu, \delta, k}g\left(\frac{\nu}{b'_{\nu, \delta, k}}\right)=k-\frac{1}{4}+O\left(k^{-1}\right).
\end{equation}
\end{lemma}

\begin{proof}
For any $0\leq \nu\leq V$, we  apply \eqref{y1} with $c=1$ to study $y_{\nu,\delta}'(x)$. We first fix a sufficiently large constant $C_V$ such that  if  $x>C_V$, then the error term $O_c(x^{-1})$  in \eqref{y1} is less than $10^{-10}$. Note that if $k$ is sufficiently large,  the interval
\begin{equation}\label{interval}
  [(k+\nu/2-1/2)\pi,(k+\nu/2+1/2)\pi)
\end{equation}
is contained in $(C_{V},\infty)$. Since
\begin{equation*}
  h_{\nu}((k+\nu/2\pm1/2)\pi)=k\pm\frac{1}{2}+O(\nu^2/k),
\end{equation*}
there exists $(a_k,b_k)\subset [(k+\nu/2-1/2)\pi,(k+\nu/2+1/2)\pi)$ such that
\begin{equation*}
  h_{\nu}((a_k,b_k))=(k-3/8,k+1/8).
\end{equation*}
Then from \eqref{y1}, it is easy to see that
\begin{equation*}
y_{\nu,\delta}'(a_k)y_{\nu,\delta}'(b_k)<0.
\end{equation*}
By the intermediate value theorem there exists at least one zero of $y_{\nu,\delta}'(x)$ in $(a_k, b_k)$.

On the other hand,  it follows from \cite[Lemmas 3.1 and 3.2]{FLPS:2024} that if $k$ is sufficiently large, there exists exactly one zero, which is $b_{\nu,\delta,k}'$, in the interval \eqref{interval}.
%, since $y_{\nu,\delta}'(x)$ has exactly $k$ zeros  in  the interval $(0, (k+\nu/2+1/2)\pi)$ and $k-1$ zeros  in  the interval $(0, (k+\nu/2-1/2)\pi)$ (see \cite[Lemmas 3.1 and 3.2]{FLPS:2024}).
Then from \eqref{y1}, we have
\begin{equation*}
  h_{\nu}(b_{\nu,\delta,k}')=k-\frac{1}{4}+O\left((b_{\nu,\delta,k}')^{-1}\right).
\end{equation*}
Notice that $h_{\nu}(b_{\nu,\delta,k}')\in (k-3/8,k+1/8)$ implies that $b_{\nu,\delta,k}'\gtrsim k$, hence \eqref{eqsmallnu}. This completes the proof.
\end{proof}

%\begin{proposition}
%  If $\nu\ge |\delta|$, then
%  \begin{equation}\label{lowerbound}
%    b_{\nu,\delta,k}'\gtrsim k+\nu.
%  \end{equation}
%\end{proposition}

\begin{theorem}
For any fixed $\nu\geq 0$ and  $\delta\in\mathbb{R}$, if  $k$ is sufficiently large, then
\begin{equation}\label{a1}
a'_{\nu, \delta, k}=\left(k+\frac{\nu}{2}-\frac{3}{4} \right)\pi +O_{\nu}\left(k^{-1}\right),
\end{equation}
and
\begin{equation}\label{b1}
b'_{\nu, \delta, k}= \left(k+\frac{\nu}{2}-\frac{1}{4} \right)\pi +O_{\nu}\left(k^{-1} \right).
\end{equation}
\end{theorem}
\begin{remark}
See \cite[Theorem 4.11]{GJWY:2024} for \eqref{a1} when $\nu\geq |\delta|$. Due to space constraints, a direct proof was omitted there. For completeness, we provide   asymptotics of $J_{\nu}$ and $J_{\nu}'$  in the appendix. The asymptotics  \eqref{a1} for $a_{\nu,\delta,k}'$ can be derived similarly to that for $b_{\nu,\delta,k}'$.
\end{remark}
\begin{proof}
From Lemmas \ref{largenu} and \ref{smallnu}, we have
\begin{equation}\label{b11}
  b'_{\nu, \delta, k}g\left(\frac{\nu}{b'_{\nu, \delta, k}}\right)=k-\frac{1}{4}+O(k^{-1})
\end{equation}
for any fixed $\nu\ge 0$ and sufficiently large $k$. Meanwhile, by the Taylor expansion of the function $g$ at $0$, we have
\begin{equation}\label{b12}
  b'_{\nu, \delta, k}g\left(\frac{\nu}{b'_{\nu, \delta, k}}\right)=\frac{1}{\pi}\left(b'_{\nu, \delta, k}-\frac{\nu}{2}\pi+O_{\nu}((b'_{\nu, \delta, k})^{-1})\right).
\end{equation}
Combining \eqref{b11} and \eqref{b12} then yields \eqref{b1}. This completes the proof.
\end{proof}

\section{proof of theorem \ref{mainresult}}
Once the one-term asymptotics \eqref{a1} and \eqref{b1}
are derived, combining them with Hankel's expansions for Bessel functions and the Lagrange inversion theorem yields our main result: the McMahon-type expansions \eqref{expansions} for  $a_{\nu,\delta,k}'$ and $b_{\nu,\delta,k}'$.
\begin{proof}[Proof of Theorem \ref{mainresult}]
We only prove \eqref{expansions} for $b'_{\nu, \delta, k}$, as the proof for $a'_{\nu, \delta, k}$ is similar. For any fixed $\nu\geq 0$ and sufficiently large $x$,  applying Hankel's asymptotic expansions for $Y_{\nu}(x)$ and $Y_{\nu}'(x)$ (see \cite[P. 364-365]{abram:1972}) yields
\begin{equation}\label{zeros1}
\begin{split}
\sqrt{\frac{\pi x}{2}}x^{\delta}y_{\nu,\delta}'(x)\!=\!\left(\!R(\nu,x)\!-\!\frac{\delta} {x}Q(\nu,x)\!\right)\cos\chi
\!-\!\left(\!S(\nu,x)\!+\!\frac{\delta}{x}P(\nu,x)\!\right)\sin\chi,
\end{split}
\end{equation}
where $\chi=x-(\nu/2+1/4)\pi$, and
\begin{equation*}
P(\nu,x)\thicksim1-\frac{A_2(\nu)}{x^2}
+\frac{A_4(\nu)}{x^4} -\frac{A_6(\nu)}{x^6}+\cdots,
\end{equation*}
\begin{equation*}
Q(\nu,x)\thicksim\frac{A_1(\nu)}{x}-\frac{A_3(\nu)}{x^3}+\frac{A_5(\nu)}{x^5}-\frac{A_7(\nu)}{x^7}+\cdots,
\end{equation*}
\begin{equation*}
R(\nu,x)\thicksim1-\frac{\mu+15}{\mu-3^2}\frac{A_2(\nu)}{x^2}
+\frac{\mu+63}{\mu-7^2}\frac{A_4(\nu)}{x^4} -\frac{\mu+143}{\mu-11^2}\frac{A_6(\nu)}{x^6}+\cdots,
\end{equation*}
\begin{equation*}
S(\nu,x)\!\thicksim\!\frac{\mu\!+\!3}{\mu\!-\!1}\frac{A_1(\nu)}{x}-\frac{\mu\!+\!35}{\mu\!-\!5^2}\frac{A_3(\nu)}{x^3}
+\frac{\mu\!+\!99}{\mu\!-\!9^2}\frac{A_5(\nu)}{x^5}-\frac{\mu\!+\!195}{\mu\!-\!13^2}\frac{A_7(\nu)}{x^7}+\cdots,
\end{equation*}
with  $\mu=4\nu^2$ and
\begin{equation}\label{def-A}
A_{s}(\nu)=\frac{(\mu-1)(\mu-3^2)\cdots(\mu-(2s-1)^2)}{s! 8^s}.
\end{equation}

Taking $x=b_{\nu,\delta,k}'$ in \eqref{zeros1} for sufficiently large $k$, dividing the resulting equation by $\sin\left(b_{\nu,\delta,k}'-\left({\nu}/{2}+{1}/{4}\right)\pi\right)$, and then applying the arccotangent function yields
\begin{equation}\label{2}
b_{\nu,\delta,k}'=\left(k+{\nu}/{2}-{3}/{4}\right)\pi+\mathrm{arccot}\,T(\nu,b_{\nu,\delta,k}') ,
\end{equation}
where
\begin{equation*}
  T(\nu,b_{\nu,\delta,k}')=\frac{S(\nu,b_{\nu,\delta,k}')+\frac{\delta}{b_{\nu,\delta,k}'}P(\nu,b_{\nu,\delta,k}')}{R(\nu,b_{\nu,\delta,k}'),
 -\frac{\delta}{b_{\nu,\delta,k}'}Q(\nu,b_{\nu,\delta,k}')}
\end{equation*}
and here we have used the observation from   \eqref{b1}  that
\begin{align*}
b_{\nu,\delta,k}'-\left({\nu}/{2}+{1}/{4}\right)\pi=(k-1/2)\pi+O_{\nu}\left(k^{-1}\right).
\end{align*}
%then
%\begin{align*}
% \cot\left(b_{\nu,\delta,k}'-\left({\nu}/{2}+{1}/{4}\right)\pi\right)
% =\frac{S(\nu,b_{\nu,\delta,k}')+\frac{\delta}{b_{\nu,\delta,k}'}P(\nu,b_{\nu,\delta,k}')}{R(\nu,b_{\nu,\delta,k}')
% -\frac{\delta}{b_{\nu,\delta,k}'}Q(\nu,b_{\nu,\delta,k}')},
%\end{align*}
%and applying the arccotangent function yields
Applying  the Taylor expansion of the function $\mathrm{arccot} $ at $0$ to \eqref{2}, we obtain
\begin{equation}\label{3}
\begin{split}
b_{\nu,\delta,k}'\thicksim &\left(k+\frac{\nu}{2}-\frac{1}{4}\right)\pi-  T(\nu,b_{\nu,\delta,k}')
 +\frac{ \left(T(\nu,b_{\nu,\delta,k}')\right)^3}{3}-\frac{ \left(T(\nu,b_{\nu,\delta,k}')\right)^5}{5}\\
 &\quad +\frac{\left(T(\nu,b_{\nu,\delta,k}')\right)^7}{7} +\cdots.
 \end{split}
\end{equation}
A straightforward calculation  shows that
\begin{equation*}
 T(\nu,b_{\nu,\delta,k}')\thicksim \frac{B_1(\nu)}{b_{\nu,\delta,k}'}+\frac{B_3(\nu)}{b_{\nu,\delta,k}'^3}+\frac{B_5(\nu)}{b_{\nu,\delta,k}'^5}+\frac{B_7(\nu)}{b_{\nu,\delta,k}'^7}+\cdots,
\end{equation*}
where
\begin{equation*}
  B_1(\nu,\delta)=\frac{\mu+3}{\mu-1}A_1(\nu)+\delta,
\end{equation*}
\begin{equation*}
\begin{split}
& B_3(\nu,\delta)=\\
&\quad  -\frac{\mu\!+\!35}{\mu\!-\!5^2}A_3(\nu)-\delta A_2(\nu)\!+\!\left(\!\frac{\mu\!+\!3}{\mu\!-\!1}A_1(\nu)+\delta\!\right)
  \!\!\left(\!\frac{\mu\!+\!15}{\mu\!-\!3^2}A_2(\nu) +\!\delta A_1(\nu)\!\right)\!,
  \end{split}
\end{equation*}
\begin{equation*}
\begin{split}
&  B_5(\nu,\delta)=\\
&\quad \frac{\mu\!+\!99}{\mu\!-\!9^2}A_5(\nu)\!+\!\delta A_4(\nu)\!-\!\!\left(\!\frac{\mu\!+\!3}{\mu\!-\!1}A_1(\nu)\!+\!\delta\!\right)\!
\left(\!\frac{\mu\!+\!63}{\mu\!-\!7^2}A_4(\nu)\!+\!\delta A_3(\nu)\!\!\right) +\\
&\quad\left(\!\frac{\mu\!+\!15}{\mu\!-\!3^2}A_2(\nu)+\delta A_1(\nu)\!\right)\!
\Bigg(\!\!-\frac{\mu\!+\!35}{\mu\!-\!5^2}A_3(\nu)\!-\!\delta A_2(\nu)
\!+\!\left(\!\frac{\mu\!+\!3}{\mu\!-\!1}A_1(\nu)\!+\!\delta\!\right)\\
&\quad\left(\!\frac{\mu\!+\!15}{\mu\!-\!3^2}A_2(\nu) \!+\!\delta\! A_1(\nu)\!\right)\!\!\Bigg),
\end{split}
\end{equation*}
and
\begin{equation*}
\begin{split}
 & B_7(\nu,\delta)=\\
 &\quad-\frac{\mu\!+\!195}{\mu\!-\!13^2}A_7(\nu)\!-\! \delta A_6(\nu) \!+\!\left(\!\frac{\mu \!+\!3}{\mu\!-\!1}A_1(\nu)\!+\!\delta\!\right)\!
  \left(\!\frac{\mu\!+\!143}{\mu\!-\!11^2}A_6(\nu)\!+\!\delta A_5(\nu)\!\right)\\
 &\quad -\!\left(\!\frac{\mu\!+\!63}{\mu\!-\!7^2}A_4(\nu)\!+\!\delta A_3(\nu)\!\!\right)\!
\Bigg(\!\!-\!\frac{\mu\!+\!35}{\mu\!-\!5^2}A_3(\nu)\!-\!\delta A_2(\nu)
\!+\!\!\left(\!\frac{\mu\!+\!3}{\mu\!-\!1}A_1(\nu)\!+\!\delta\!\right)\\
& \quad\left(\!\frac{\mu\!+\!15}{\mu\!-\!3^2}\!A_2(\nu) \!+\!\delta\! A_1\!(\nu)\!\!\right)\!\!\!\Bigg)\!
\!+\!\!\left(\!\frac{\mu\!+\!15}{\mu\!-\!3^2}A_2\!(\nu)\!+\!\delta\! A_1\!(\nu)\!\!\right)\!\!\Bigg(\!\frac{\mu\!+\!99}{\mu\!-\!9^2}A_5(\nu)\!+\!\delta\! A_4(\nu)\\
&\quad-\!\!\left(\!\frac{\mu\!+\!3}{\mu\!-\!1}A_1(\nu)\!+\!\delta\!\right)\left(\!\frac{\mu\!+\!63}{\mu\!-\!7^2}A_4(\nu)\!+\!\delta A_3(\nu)\!\right)\!
+\!\left(\!\frac{\mu\!+\!15}{\mu\!-\!3^2}A_2(\nu)\!+\!\delta A_1(\nu)\!\right)
\\
&\quad \Bigg(\!\!\!-\!\frac{\mu\!+\!35}{\mu\!-\!5^2}A_3(\nu)\!-\!\delta A_2(\nu)
\!+\!\!\left(\!\frac{\mu\!+\!3}{\mu\!-\!1}A_1(\nu)\!+\!\delta\!\right)
\!\!\left(\!\frac{\mu\!+\!15}{\mu\!-\!3^2}A_2(\nu) \!+\!\delta\! A_1(\nu)\!\right)\!\!\Bigg)\!\!\Bigg).
  \end{split}
\end{equation*}
Thus, \eqref{3} becomes
\begin{equation}\label{4}
\begin{split}
&b_{\nu,\delta,k}'\thicksim\\
&\quad \beta'+\frac{-B_1(\nu,\delta)}{b_{\nu,\delta,k}'} + \frac{-B_3(\nu,\delta)\!+\!\frac{1}{3}(B_1(\nu,\delta))^3}{b_{\nu,\delta,k}'^3}+\frac{U(\nu,\delta)}{b_{\nu,\delta,k}'^5}+\frac{\tilde U(\nu,\delta)}{b_{\nu,\delta,k}'^7}+\cdots\!,
\end{split}
\end{equation}
where $\beta'=\left(k+\frac{\nu}{2}-\frac{1}{4}\right)\pi$,
\begin{equation*}
U(\nu,\delta)= -B_5(\nu,\delta)+(B_1(\nu,\delta))^2B_3(\nu,\delta)-\frac{1}{5}(B_1(\nu,\delta))^5
\end{equation*}
and
\begin{equation*}
\begin{split}
\tilde U(\nu,\delta)=&-B_7(\nu)+(B_1(\nu,\delta))^2B_5(\nu,\delta)+B_1(\nu,\delta)(B_3(\nu,\delta))^2\\
&-(B_1(\nu,\delta))^4B_3(\nu,\delta)+\frac{1}{7}(B_1(\nu,\delta))^7.
\end{split}
\end{equation*}
Applying  the Lagrange inversion theorem to \eqref{4} and using \eqref{def-A} (the definition of $A_s(\nu)$) yields the desired expansion for $b_{\nu,\delta,k}'$ in \eqref{expansions}. This completes the proof of Theorem \ref{mainresult}.
\end{proof}

\appendix

\section{related asymptotics of the Bessel functions $J_{\nu}(x)$ and $J_{\nu}'(x)$}
The following two lemmas concern the asymptotics of Bessel functions $J_{\nu}(x)$ and $J_{\nu}'(x)$ under different assumptions on $x$ and $\nu$.
The proofs of these asymptotics  are similar to those for $Y_{\nu}(x)$ and $Y_{\nu}'(x)$ in Lemmas \ref{lemma1} and  \ref{lemma2}. See also  \cite[ Appendix A]{GMWW:2021}  and \cite[Lemmas 2.1 and 2.2]{Guo} for the case of $J_{\nu}(x)$.

\begin{lemma}\label{lemmaa1}
For any $c>0$ and $\nu\ge 0$, if $x\ge \max\{(1+c)\nu,10\}$ then
 \begin{equation*}
J_{\nu}(x)=\frac{2^{1/2}}{\pi^{1/2}\left(x^2-\nu^2\right)^{1/4} }\left(\sin\left( \pi
xg\left(\frac{\nu}{x}\right)+\frac{\pi}{4}\right)+O_{c}\left(x^{-1}\right)\right),
\end{equation*}
\begin{equation*}
J_{\nu}'(x)=\frac{2^{1/2}\left(x^2-\nu^2\right)^{1/4}}{\pi^{1/2}x}\left(\sin\left( \pi
xg\left(\frac{\nu}{x}\right)+\frac{3\pi}{4}\right)+O_{c}\left(x^{-1}\right)\right).
\end{equation*}
\end{lemma}

\begin{lemma}\label{lemmaa2}
 For any $c>0$ and sufficiently large $\nu$, if $\nu<x<(1+c)\nu$ then
 \begin{equation*}
J_{\nu}(x)=\frac{(12\pi xg(\nu/x) )^{1/6}}{(x^2-\nu^2)^{1/4}}\left(\mathrm{Ai}\!\left(\!-\!\left(\frac{3\pi}{2} xg\!\left(\frac{\nu}{x}\right)\!\right)^{\!2/3}\right)+O\left(\nu^{-4/3}\right)\right),
\end{equation*}
 \begin{equation*}
J_{\nu}'(x)=\frac{-2\left(x^2-\nu^2\right)^{1/4}}{(12\pi xg\left(\nu/x \right))^{1/6}x}
	\!\left(\!\mathrm{Ai}'\!\left(\!-\!\left(\!\frac{3\pi}{2} x g\!\left(\frac{\nu}{x} \right)\!\right)^{\!2/3}\right)\!+O\left(\nu^{-2/3}\right)\!\right)
\end{equation*}
  when $xg(\nu/x)\leq C$ for some large constant $C>0$, and
\begin{equation*}
J_{\nu}(x)=\frac{2^{1/2}}{\pi^{1/2} \left(x^2-\nu^2\right)^{1/4}} \left(\sin\left(\pi xg\!\left(\frac{\nu}{x}\right)+\frac{\pi}{4}\right)
+O\left(\!\left(xg\left(\frac{\nu}{x}\right)\right)^{\!-1}\right)\!\right),
\end{equation*}
\begin{equation*}
J_{\nu}'(x)=\frac{2^{1/2}\left(x^2-\nu^2\right)^{1/4}}{\pi^{1/2}x} \left(\sin\left(\pi xg\!\left(\frac{\nu}{x}\right)+\frac{3\pi}{4}\right)
+O\left(\!\left(xg\left(\frac{\nu}{x}\right)\right)^{\!-1}\right)\!\right)
\end{equation*}
when $xg(\nu/x)>C$. Here the implicit constants only depends on $c$ and $C$.
\end{lemma}

\subsection*{Acknowledgements}
I would like to thank Prof. Jingwei Guo for valuable advice.

%%%%%%%%%%%%%%%%%%%%%%%%%%%%%%%%%%%%%%%%%%%%%%%%%%%%%%%%%%%%%%%%%%%%%%%%%%%%%%%%
%%%%%%%%%%%%%%%%%%%%%%%%%%%%%%%%%%%%%%%%%%%%%%%%%%%%%%%%%%%%%%%%%%%%%%%%%%%%%%%%

\end{document}